\documentclass[11pt]{amsart}
\usepackage{amsmath,amsthm,amssymb,amscd,eucal}        
\usepackage{graphicx, color}
\usepackage[all]{xy}

\usepackage[skip=10pt]{caption}

\usepackage[margin=1.3in]{geometry}

\def\R{\mathbb{R}}
\def\Z{\mathbb{Z}}

\newcommand{\squeezelist}{\setlength{\itemsep}{0pt}}

\newcommand{\ext}{$\varepsilon$} 	
\newcommand{\ter}{$\tau$}			

\newcommand{\hide}[1]{}

\theoremstyle{plain}
\newtheorem{thm}{Theorem}

\newtheorem{prop}[thm]{Proposition}

\newtheorem{lem}[thm]{Lemma}
\newtheorem*{conj}{Conjecture}

\theoremstyle{definition}
\newtheorem*{defn}{Definition}
\newtheorem*{exmp}{Example}

\theoremstyle{remark}
\newtheorem*{rem}{Remark}
\newtheorem*{ack}{Acknowledgments}
\numberwithin{equation}{section}

\begin{document}

\title{Unfolding cubes: nets, packings, partitions, chords}

\author[K.\ DeSplinter]{Kristin DeSplinter}
\address{K.\ DeSplinter: University of Utah, Salt Lake City, UT 84112}
\email{desplinter.k@utah.edu}

\author[S.\ Devadoss]{Satyan L.\ Devadoss}
\address{S.\ Devadoss: University of San Diego, San Diego, CA 92110}
\email{devadoss@sandiego.edu}

\author[J.\ Readyhough]{Jordan Readyhough}
\address{J.\ Readyhough: Columbia University, New York, NY 10027}
\email{jhr2150@columbia.edu}

\author[B.\ Wimberly]{Bryce Wimberly}
\address{B.\ Wimberly: Trident Seafood Analysis, Seattle, WA 98107}
\email{bwimberly2@gmail.com}

\begin{abstract}
We show that every ridge unfolding of an $n$-cube is without self-overlap, yielding a valid net.  The results are obtained by developing machinery that translates cube unfolding into combinatorial frameworks.   Moreover, the geometry of the bounding boxes of these cube nets are classified using integer partitions, as well as the combinatorics of  path unfoldings seen through the lens of chord diagrams.
\end{abstract}



\maketitle
\baselineskip=16pt

\vspace{-.3in}
%
%
\section{Introduction} 

The study of unfolding polyhedra was popularized by Albrecht D\"urer  in the early 16th century in his influential book \emph{The Painter's Manual}. It contains the first recorded examples of polyhedral \emph{nets}, connected edge unfoldings of polyhedra that lay flat on the plane without overlap. Motivated by this, Shephard  \cite{sh1} conjectures that every convex polyhedron can be cut along certain edges and admits a net.  This claim remains tantalizingly open and has resulted in numerous areas of exploration.

We consider this question for higher-dimensional convex \emph{polytopes}. The codimension-one faces of a polytope are  \emph{facets} and its codimension-two faces are \emph{ridges}. The analog of an edge unfolding of polyhedron is the \emph{ridge unfolding} of an $n$-dimensional polytope: the process of cutting the polytope along a collection of its ridges so that the resulting (connected) arrangement of its facets develops isometrically into an $\R^{n-1}$ hyperplane.

There is a rich history of higher-dimensional unfoldings of polytopes, with the collected works of Alexandrov \cite{alex1} serving as seminal reading. In 1984, Turney \cite{tur} enumerates the 261 ridge unfoldings of the 4-cube, and in 1998, Buekenhout and Parker \cite{bupa} extend this enumeration to the other five regular convex 4-polytopes.  Both of these works focus on enumerative rather than geometric unfolding results. Miller and Pak \cite{mp} construct an  algorithm which provides an unfolding of polytopes without overlap.  However, their method allows cuts interior to facets, not just along ridges.

Our work targets ridge unfoldings of the $n$-cube.  For the $3$-cube, Figure~\ref{f:3dnets} shows the 11 different unfoldings (up to symmetry), all of which yield nets.    Section~\ref{s:nets} generalizes this into our main result:  \emph{every} ridge unfolding of the $n$-cube results in a net.     Section~\ref{s:pack} considers packing these cube nets into boxes and classifies them using integer partitions. Finally, motivated by architectural housing developments, Section~\ref{s:path} considers the combinatorics of  path unfoldings through the lens of chord diagrams.

\begin{figure}[h]
\includegraphics[width=.9\textwidth]{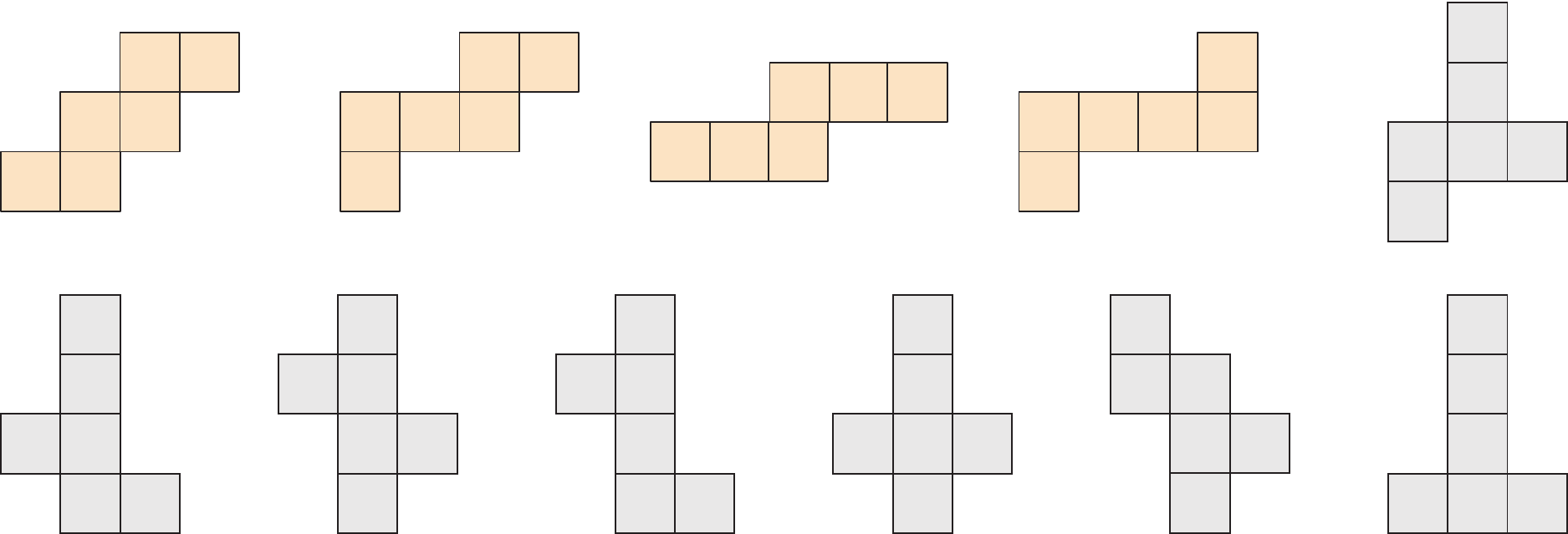}
\caption{Nets of the 3-cube, with path unfoldings highlighted.}
\label{f:3dnets}
\end{figure}

\begin{ack}
We thank Bob Connelly, Erik Demaine, Joe O'Rourke, and Shannon Starkey for guidance and encouraging conversations.  This work was partially supported from an endowment from the Fletcher Jones Foundation.
\end{ack}

%
%
\section{Rolling and Unfolding}  \label{s:nets}
\subsection{}

We explore ridge unfoldings of a convex polytope $P$ by focusing on the combinatorics of the arrangement of its facets in the unfolding. In particular, a ridge unfolding induces a tree whose nodes are the facets of the polytope and whose edges are the uncut ridges between the facets \cite{sh2}.  Indeed, this is a spanning tree in the 1-skeleton of the dual of $P$.
The dual of the $n$-cube is the $n$-orthoplex, whose 1-skeleton forms the $n$-\emph{Roberts graph}.\footnote{This graph has numerous names, including the Tur\'an graph $T(2n,n)$ and the $n$-cocktail party graph.}  The $2n$ nodes of this graph (corresponding to the $2n$ facets of the $n$-cube) can arranged on a circle so that antipodal nodes represent opposite facets of the cube.   Its edges connect every pair of nodes except for antipodal ones. Thus, unfoldings of an $n$-cube are in bijection with spanning trees of the $n$-Roberts graph.

\begin{figure}[h]
\includegraphics[width=.9\textwidth]{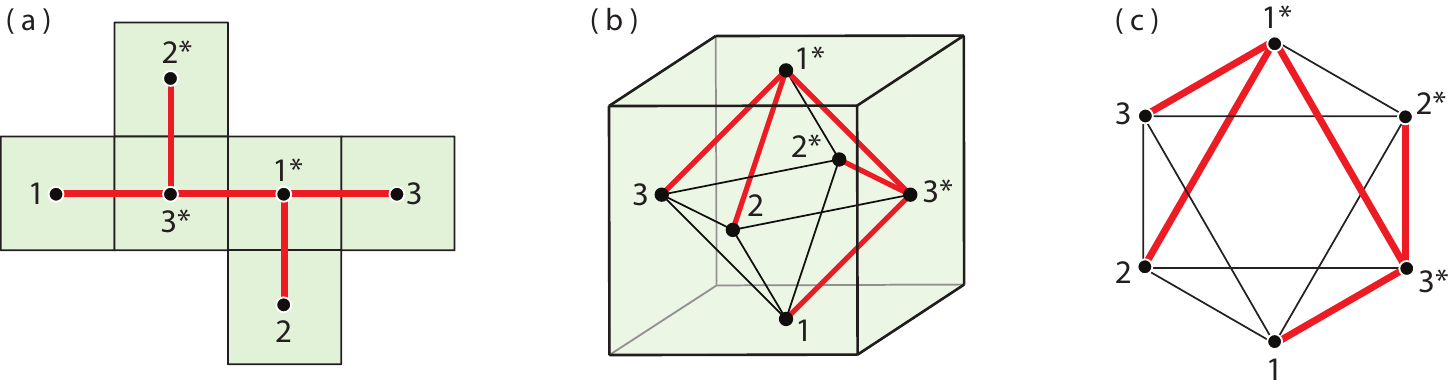}
\caption{Unfolding a 3-cube with its spanning tree on the Roberts graph.}
\label{f:cubetree}
\end{figure}

\begin{exmp}
Figure~\ref{f:cubetree}(a) considers an edge unfolding of the 3-cube with its underlying dual tree.  This appears as a spanning tree on the 1-skeleton of the octahedral dual (b), redrawn on the $3$-Roberts graph (c).
\end{exmp}

Recall that a ridge unfolding of an $n$-cube is a connected arrangement of its $2n$ facets, developed isometrically into hyperplane $\R^{n-1}$.  Begin the unfolding by choosing a (base) facet $b$ of the $n$-cube, placing it on the hyperplane.  Then the normal vector $n_b$ to $b$ becomes normal to the hyperplane. Consider an adjacent facet $c$ to $b$, and roll the cube along the ridge between these facets, with facet $c$ now landing on the hyperplane.  Figure~\ref{f:cube-rotate} shows the orthogonal projection of such a roll, with $c*$ and $b*$ corresponding to the antipodal facets of $c$ and $b$, and the marked red edge representing the ridge between $c$ and $b$.

\begin{figure}[h]
\includegraphics[width=1\textwidth]{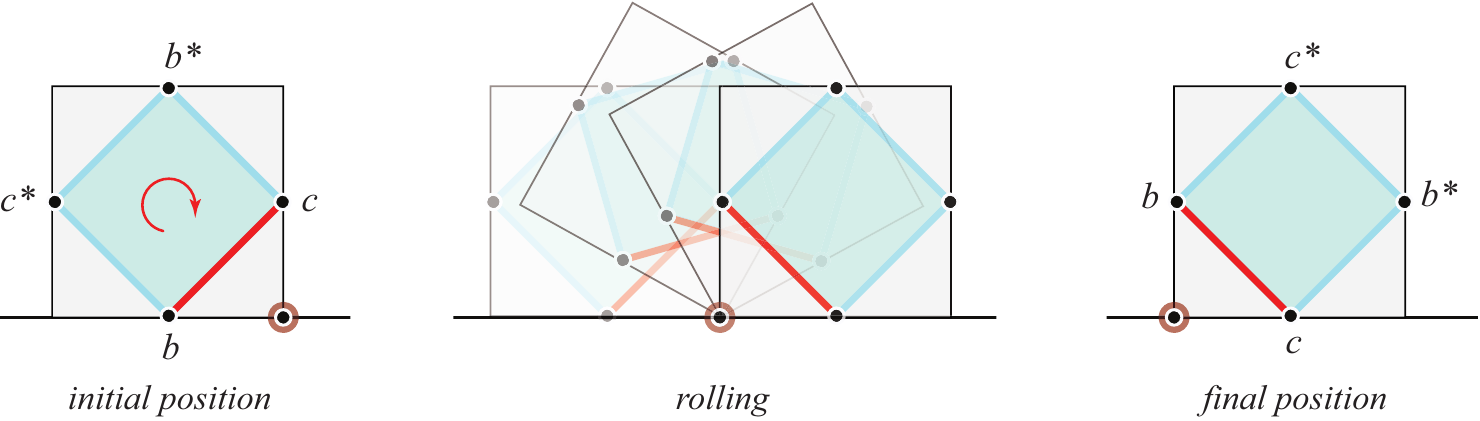}
\caption{Rolling a cube along a hyperplane.}
\label{f:cube-rotate}
\end{figure}

Since we rotate only along the plane spanned by the normal vectors $n_b$ and $n_c$, the remaining directions stay fixed in the development. This is captured combinatorially as a rotation of a subgraph of the Roberts graph: 

\begin{defn}
A \emph{roll} from base facet $b$ towards an adjacent facet $c$ rotates the four nodes $\{b, c, b*, c*\}$ of the Roberts graph along the quadrilateral (keeping the remaining nodes fixed), making $c$ the new base facet.
\end{defn}

\noindent Figure~\ref{f:5d-spin} shows an example for the $5$-cube, where the highlighted quadrilateral (depicting the roll) is invoking the colored square of Figure~\ref{f:cube-rotate}. 

\begin{figure}[h]
\includegraphics[width=\textwidth]{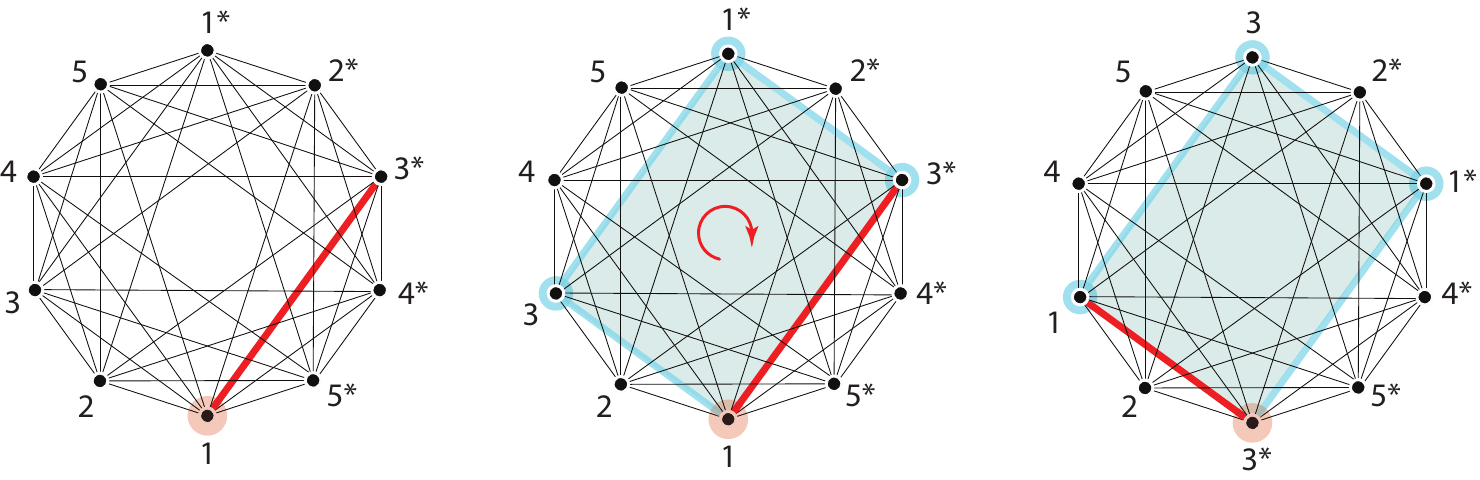}
\caption{A roll of a $5$-cube, rotating facet 1 towards facet $3^*$.} 
\label{f:5d-spin}
\end{figure}

\subsection{}

The advantage of unfolding a cube (compared to an arbitrary convex polytope) into hyperplane $\R^{n-1}$ is that its $(n-1)$-cube facets naturally tile this hyperplane. We  exploit this  by recasting the unfolding in the language of lattices. Without loss of generality, we can situate a ridge unfolding of the $n$-cube so that the centroid of each facet occupies a point of the integer lattice $\Z^{n-1}$ of $\R^{n-1}$.  To see the lattice structure manifest in the $n$-Roberts graph, we imbue the latter with a \emph{coordinate system}:  arbitrarily label the $2n-2$ edges of the $n$-Roberts graph incident to the base node with the directions
$$\{x_1, -x_1, x_2, -x_2, \dots, x_{n-1}, -x_{n-1}\} \, ,$$ 
where edges incident to antipodal nodes get opposite directions.\footnote{The isometry group of the cube, the Coxeter group of type $B$, acts simply transitively on these edges. Without loss of generality, we can choose a counterclockwise labeling of the edges in cyclic order.}   
Figure \ref{f:directions} shows examples of coordinate systems for the 3D, 4D, and 5D cases.

\begin{figure}[h]
\includegraphics[width=.9\textwidth]{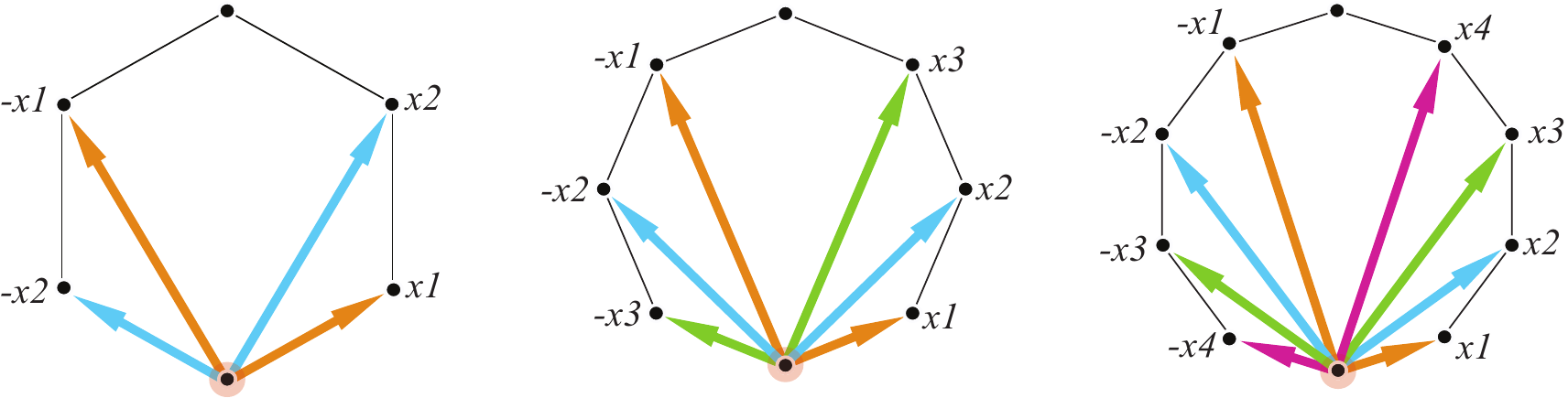}
\caption{Coordinate systems for 3D, 4D, and 5D cubes.}
\label{f:directions}
\end{figure}

These $n-1$ directions are mapped to the axes of the $\R^{n-1}$ hyperplane into which the $n$-cube unfolds.  
In particular, the $2n-2$ ridges of the $n$-cube incident to the base facet are in bijection with these coordinate directions, with opposite directions corresponding to parallel ridges of the facet.  The roll keeps track of the  combinatorics, whereas the coordinate system shows the direction of unfolding in the lattice.  This is made precise:

\begin{prop}  \label{p:roll}
Let $T$ be a spanning tree of the $n$-Roberts graph with a coordinate system.  The unfolding of the $n$-cube along $T$ into $\R^{n-1}$ can be obtained by mapping $T$ to the lattice $\Z^{n-1}$ through a sequence of rolls.
\end{prop}

\begin{proof}
Choose some base facet $b$ of $T$ and map it to some point $p_b \in \Z^{n-1}$. Let node $c$ be adjacent to $b$ along $T$ with associated direction $x$ from the coordinate system.   The roll from $b$ towards $c$ maps node $c$ to the point in $\Z^{n-1}$ that is adjacent to $p_b$ in direction $x$.   The four facet labels $\{b, c, b*, c*\}$ permute with the roll of the cube whereas the coordinate system directions are always anchored to the base facet.  In particular,  after the roll,  facet $b*$ lies in the $x$ direction with respect to the new base facet $c$, since the plane spanned by normal vectors $n_b$ and $n_c$ was rotated.

Given any node $t$ of $T$, traverse the path between $b$ and $t$ through a series of rolls as described above; this maps all the nodes of $T$ into $\Z^{n-1}$.  To obtain the unfolding of the $n$-cube, simply replace each mapped point of the lattice with an $(n-1)$-cube.
\end{proof}

\begin{exmp}
Figure \ref{f:3dpathunfold} shows an unfolding of the 3-cube along a spanning path using Proposition~\ref{p:roll}.  At each iteration, there is a roll of the Roberts graph and a direction of unfolding based on the given coordinate system.   The unfolded facets are colored white, and the unfolded ridges become dashed-lines.
Figure \ref{f:4dpathunfold} provides an example of an unfolding of the 4-cube along a spanning path.
\end{exmp}

\begin{figure}[h]
\includegraphics[width=\textwidth]{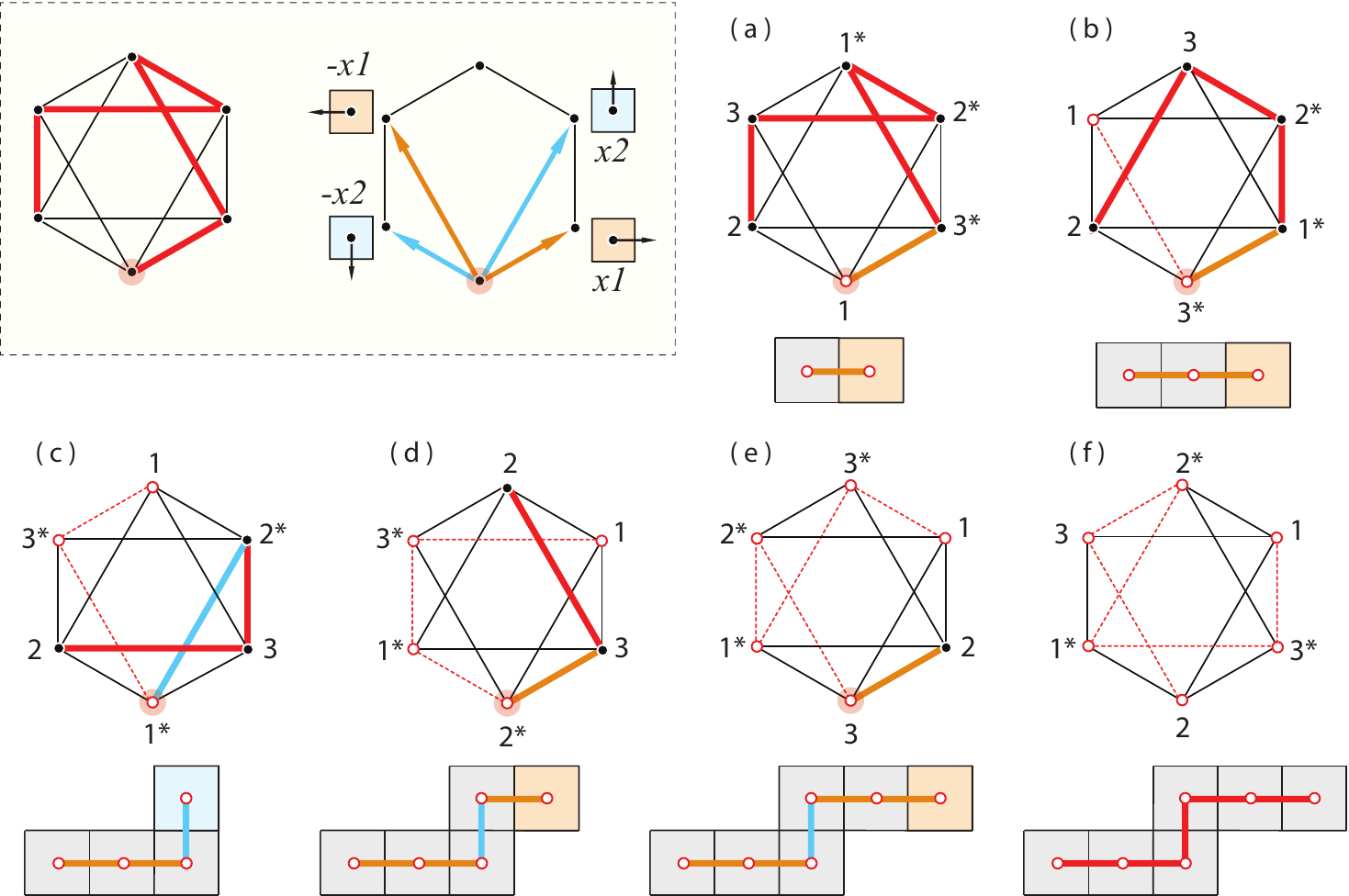}
\caption{Unfolding a 3-cube along a spanning path.}
\label{f:3dpathunfold}
\end{figure}

\begin{figure}[h]
\includegraphics[width=\textwidth]{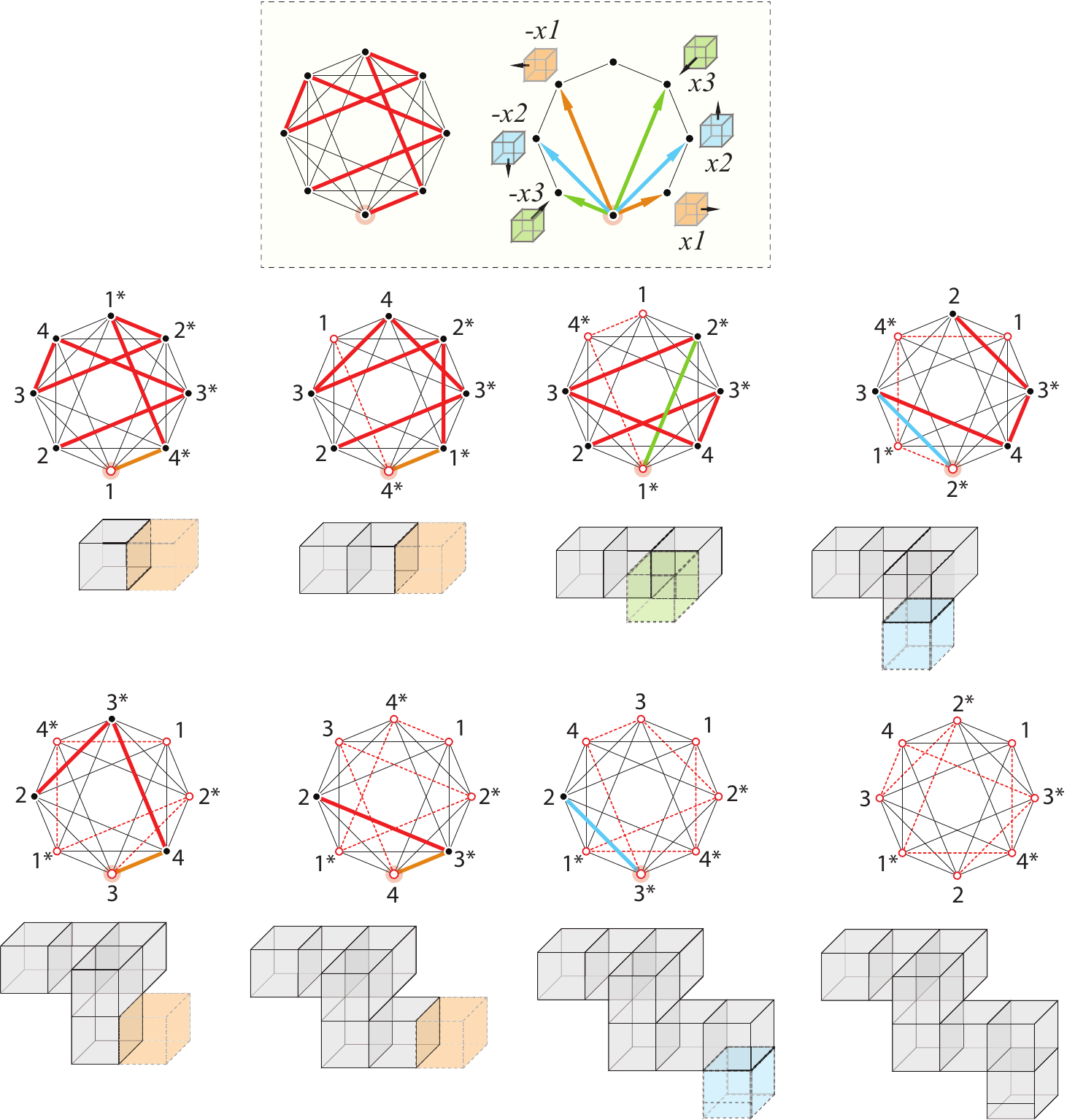}
\caption{Unfolding a 4-cube along a spanning path.}
\label{f:4dpathunfold}
\end{figure}

\subsection{}

Up to symmetry, there are 11 unfoldings of the 3-cube and 261 unfoldings of the 4-cube, all of them  nets. We now show that this property extends to all dimensions, that every unfolding of the $n$-cube results in a valid net.  

\begin{lem} \label{l:uturn}
Let $T$ be a spanning tree of the $n$-Roberts graph with a coordinate system.  If direction $x$ is used in the unfolding along some path of $T$, direction $-x$ will not be used in the unfolding along this path.
\end{lem}

\begin{proof}
Assume we roll along a path in the $x$ direction, moving the current base facet $b$ into the $-x$ direction. Since  $b$ has now been visited, it cannot be used again in the unfolding. Thus, the only way a roll along direction $-x$ can occur is if $b$ is rotated out of that direction.  However, the only moves that can displace  $b$ are rolls along the $x$ and $-x$ directions.  The latter is not possible and the former simply replaces $b$ with another visited facet, continuing to obstruct motion in the $-x$ direction.
\end{proof}

\begin{exmp}
Figure~\ref{f:3dpathunfold}(ab) shows an example where the first roll is in direction $x_1$, moving facet 1 into the $-x_1$ position, and facet $3^*$ into the base position.  Since facet $1$ has been visited, rolling in direction $-x_1$ is restricted. Another roll in Figure~\ref{f:3dpathunfold}(bc) displaces $1$ but simply replaces it with facet $3^*$, which has now been visited.
\end{exmp}

\begin{thm} 
Every ridge unfolding of an $n$-cube yields a net.
\end{thm}

\begin{proof}
Consider an unfolding of the $n$-cube, given by a spanning tree $T$ on the $n$-Roberts graph.
By Lemma~\ref{l:uturn}, antipodal directions will never appear in unfolding of paths.  Thus, as the combinatorial distance between any two nodes of a path along the spanning tree $T$ increases, the Euclidean distance of their respective facets in the hyperplane $\R^{n-1}$ (under the mapping to the integer lattice from Proposition~\ref{p:roll}) strictly increases.  Since the facets in the unfolding along any path of $T$ do not overlap, the unfolding of the entire tree $T$ results in a net.
\end{proof}

%
%
\section{Packings and Partitions} \label{s:pack}
\subsection{}

Having unfolded cubes into their nets, we now turn to packing these nets into boxes.  A \emph{box} (or \emph{orthotope}) is the Cartesian product of intervals, and the \emph{bounding box} of a net is the smallest box containing the net, with box sides parallel to the ridges of the net. 

\begin{defn}
An $n$-\emph{cube partition} is an integer partition of $3n-2$ into $n-1$ parts, where each part is at least two.
\end{defn}

\begin{exmp} 
Figure~\ref{f:4dunfolds} displays four spanning trees of the 4-cube and  their corresponding nets in bounding boxes. Notice that the dimensions of each bounding box form a $4$-cube partition.  In particular, these are all the possible $4$-cube partitions.  The following result claims that all 261 nets of the 4-cube must fit into one of these four boxes.
\end{exmp}

\begin{figure}[h]
\includegraphics[width=\textwidth]{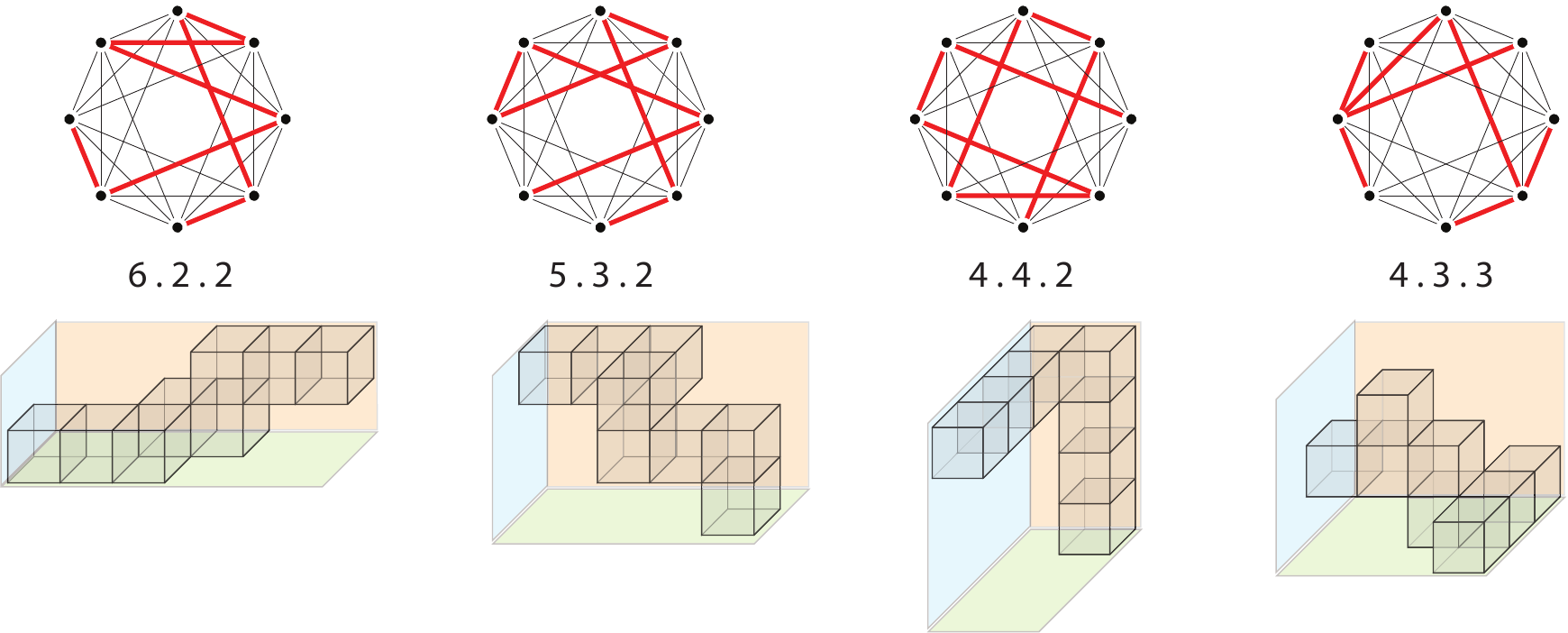}
\caption{Spanning trees and bounding boxes of 4-cube nets.}
\label{f:4dunfolds}
\end{figure}

\begin{thm} \label{t:pack}
For every net of an $n$-cube, the dimensions of its bounding box is an $n$-cube partition.
\end{thm}

\begin{proof}
Each net of the $n$-cube has $2n$ facets that need to be unfolded in $\R^{n-1}$.  Since each facet is an $(n-1)$-cube, the placement of the first facet in the unfolding contributes $n-1$ to the bounding box number of the net, one for each of its $n-1$ dimensions.  We show that each of the remaining $2n-1$ facets of the unfolding increases the bounding box number by exactly 1, resulting in a total box number of $1 \cdot (n-1) + (2n-1) \cdot 1 = 3n-2$.  

Suppose (by contradiction) that in the unfolding, the roll from facet $b$ to adjacent facet $c$ in direction $x$ does not increase the bounding box number of the current net.  Assume the ridge between $b$ and $c$ is supported by some hyperplane $\mathcal H$ of $\R^{n-1}$.  Since the box number did not increase,  there must be another facet (call it $d$) in the current unfolding that lies on the same side of  hyperplane $\mathcal H$ as $c$.  Thus, the unfolding of the path between facets $c$ and $d$ must have crossed $\mathcal H$ at least twice, moving along $x$  in both the positive and negative directions, contradicting Lemma~\ref{l:uturn}.

Finally, it needs to be shown that our cube will roll in all $n-1$ unfolding dimensions (satisfying the requirement that each part of a cube partition is a least two).  But the cube net is a spanning tree of the Roberts graph, with the unfolding forced to visit all the nodes.  And such visits can only be accomplished by rolling along each of the $n-1$ distinct directions.
\end{proof}

\subsection{}

The converse of Theorem~\ref{t:pack} also holds: given an integer partition of $3n-2$ into $n-1$ parts, there exists an unfolding of an $n$-cube whose bounding box dimensions match the partition.  The remainder of this section is devoted to proving this result. As discussed earlier, the placement of the first facet in the unfolding of the $n$-cube contributes $n-1$ to the bounding box number. Thus, the cube partition can be reinterpreted as an integer partition of $2n-1$ (the remaining facets) into $n-1$ parts (the possible directions), with each part at least one.  For such a partition, our task is to find a sequence of rolls along the $n-1$ directions so that the $2n-1$ facets are unfolded into their respective partitioned directions.
Without loss of generality, we consider rolls only in the positive directions.

In order to construct cube unfoldings for such partitions, we reinterpret the Roberts graph as a token sliding game, with Figure~\ref{f:slidegame} serving as a Rosetta stone.  Consider the first column of this figure, where the $n$-Roberts graph on top is unraveled below into a game board with $n-1$ slides (appropriately color-coded). Here, the base node of the Roberts graph is  replaced by our given partition, one for each direction, with the $2n-1$ positions represented by black tokens.  The goal of this game is to move these tokens into the $2n-1$ empty slots on the game board above by a sequence of slides, corresponding to rolls of the Roberts graph. 

\begin{figure}[h]
\includegraphics[width=\textwidth]{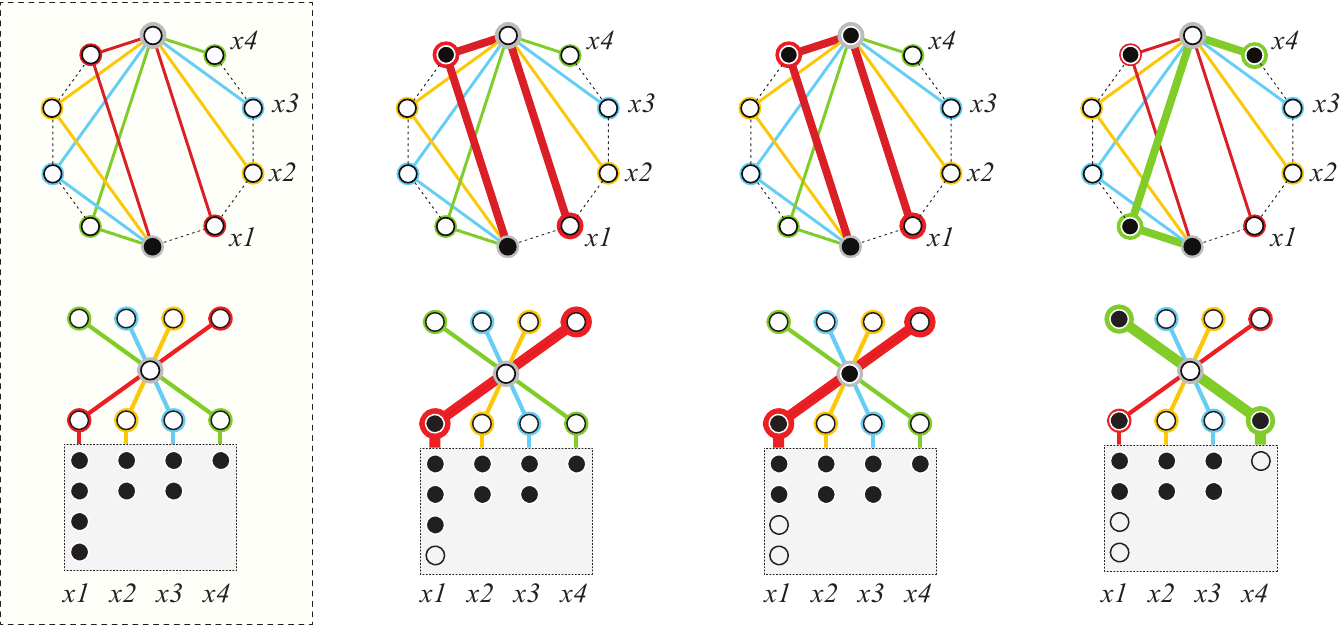}
\caption{Rolls on the Roberts graphs reinterpreted as a token sliding game.}
\label{f:slidegame}
\end{figure}

The top row of Figure~\ref{f:slidegame} shows a $5$-cube rolling twice in the $x_1$ direction, followed by a roll in the $x_4$ direction, and a roll in the $x_3$ direction. 
The bottom row displays the corresponding tokens moving along their appropriate slides, leaving the partition box and occupying empty slots on the game board above.  The features of the token game are inherited from the properties of rolls:

\begin{enumerate} \squeezelist
\item Each roll of the Roberts graph in a particular direction slides \emph{all} the tokens along that direction one place up. 
\item When a token reaches the end of its slide (eg, direction $x_4$, as displayed by the fourth column of Figure~\ref{f:slidegame}), it can no longer use that direction.
\item The antipode to the base (topmost on the Roberts graph) acts as a \emph{transfer point}, moving tokens from one directional slide into another.
\end{enumerate}

\begin{thm} \label{t:party}
For any $n$-cube partition, there exists a path unfolding of an $n$-cube whose bounding box dimensions matches the partition.
 \end{thm}

\begin{proof}
We provide an unfolding algorithm by rolling along directions satisfying a given partition.  Parts in the partition with more than one token are called \emph{towers}, whereas parts with exactly one token are dubbed \emph{singletons}.  Begin by decomposing the $2n-1$ tokens into four groups:
\begin{enumerate} \squeezelist
\item The set $S$ of tokens in the singletons.
\item The set $B$ of bottom tokens in each tower. 
\item The set $T$ of top tokens in each tower. 
\item The remaining set $M$ of (middle) tokens.
\end{enumerate}
It follows that $|T|  =  |B|  = (n-1) - |S|$
 and
$|M|  =  (2n-1) - |T| - |B| - |S|  =  |S| + 1.$
Figure~\ref{f:ordering} shows two distinct partitions of 15 tokens into 7 parts (when $n=8$), labeled according to the terminology above. 

\begin{figure}[h]
\includegraphics[width=.9\textwidth]{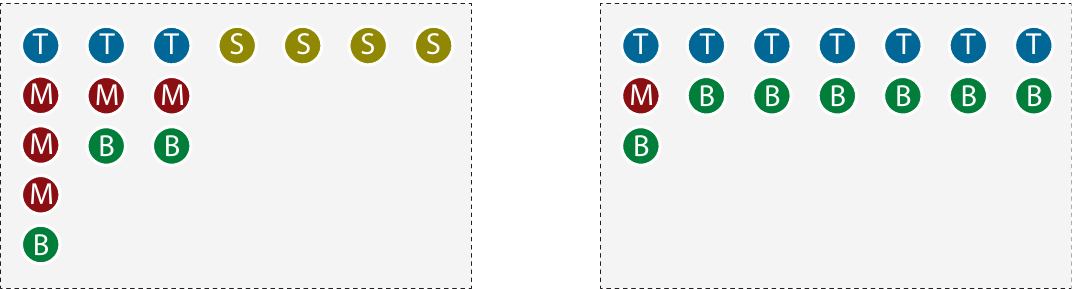}
\caption{Two partitions when $n=8$, with 15 elements into 7 parts.}
\label{f:ordering}
\end{figure}

\noindent Our algorithm is broken into three steps:

\medskip
 {\bf Step 1}: \emph{Perform one slide in each direction of a token from $B$.}  This is possible since the transfer point is empty; see Figure~\ref{f:solution}(abc).

\medskip
{\bf Step 2}: \emph{Perform alternating slides between tokens from $M$ and $S$, starting and ending with $M$, until all such tokens depleted.}  This is well-defined since $|M| = |S| + 1$. Since the first position on the game board along any element of $M$ already contains a token from Step 1, a slide along its direction moves this token into the transfer point; see Figure~\ref{f:solution}(d).  Now, sliding a token of $S$ fills the first and last positions along this directional track with tokens, making this direction unusable; see Figure~\ref{f:solution}(e).  This is ideal, for $S$ contains only one token in each direction.  After alternating between $M$ and $S$, depleting all elements of $S$, slide one final time along the last element of $M$, loading a token onto the transfer point; see Figure~\ref{f:solution}(f). 

\medskip
{\bf Step 3}: \emph{Perform one slide in each direction of a token from $T$.}  Each slide moves the token of the transfer point to the end of the track, which replenishing the transfer point with another token. This fills all the positions, as these are the final elements in each tower; see Figure~\ref{f:solution}(ghi).
\end{proof}

\begin{figure}[h]
\includegraphics[width=\textwidth]{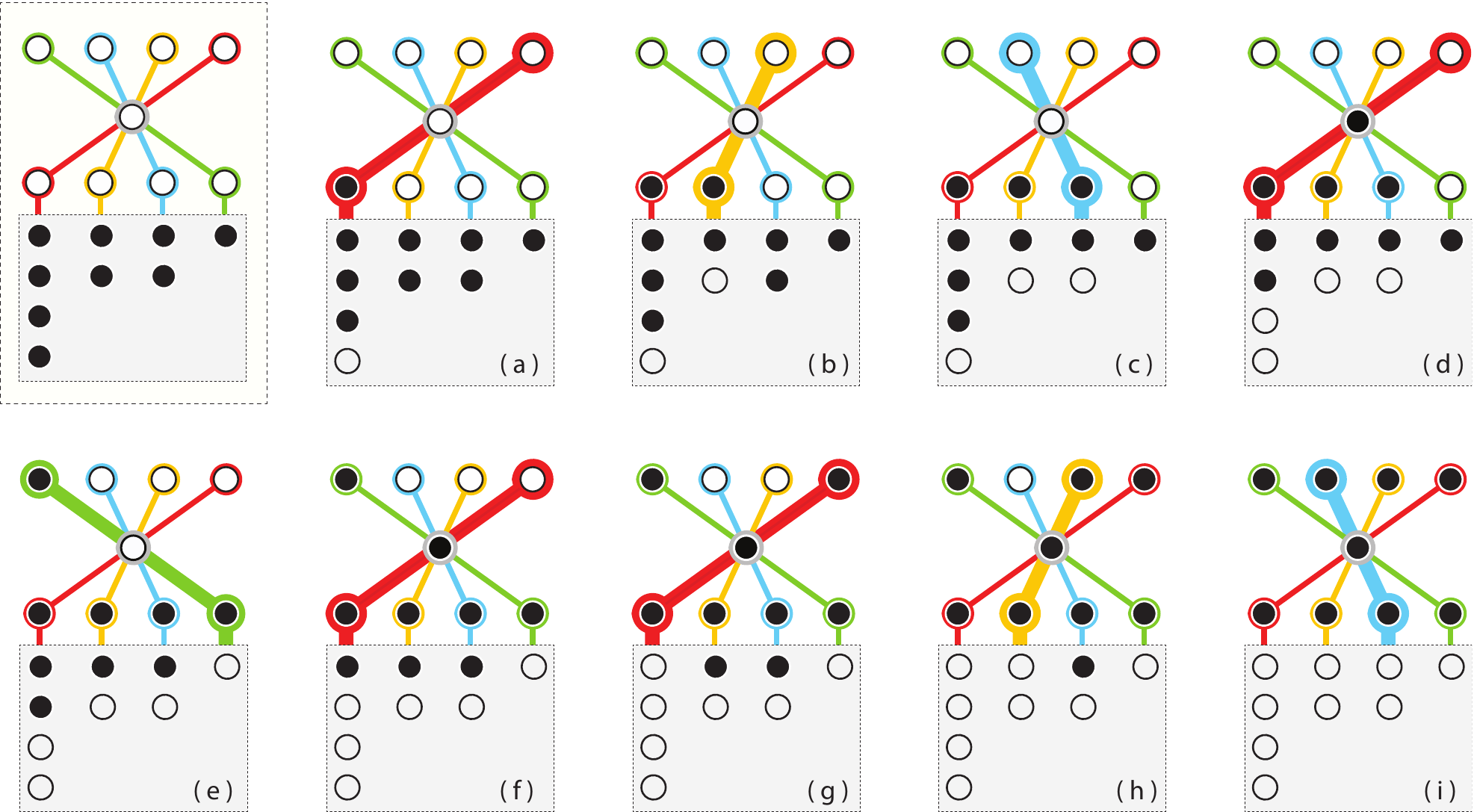}
\caption{Realizing the partition through slides.}
\label{f:solution}
\end{figure}

\begin{rem}
Theorem~\ref{t:party} shows that the $n$-cube can be unfolded into extremes: a long, thin $2 \times \dots \times 2 \times (n+2)$ box and a cubelike $3 \times \dots \times 3 \times 4$ box, with a spectrum of sizes in between.  It would be interesting to explore the \emph{distribution} of cube partitions over all possible unfoldings of the $n$-cube.
\end{rem}

\begin{rem}
Up to symmetry, there are 11 nets of the $3$-cube and 261 nets of the $4$-cube.  For a general $n$-cube, it is an open problem to enumerate its distinct nets.   The theorem above provides a (very weak) lower bound to this problem.
\end{rem}

%
%
\section{Spanning Paths and Cycles} \label{s:path}
\subsection{}

This section is devoted to exploring path unfoldings of cubes. Reconsider the four highlighted path nets from Figure~\ref{f:3dnets}, redrawn below in Figure~\ref{f:4paths}.
The first three cases (abc) are obtained by deleting an edge of a spanning cycle.\footnote{These spanning paths and cycles are sometimes called \emph{Hamiltonian} paths and cycles.} In particular, the middle two (bc) arise from the same cycle, differentiated only by the choice of deleted edge.  In contrast, the last path (d) does not arise from a spanning cycle, with endpoints on antipodal positions of the Roberts graph.

\begin{figure}[h]
\includegraphics[width=\textwidth]{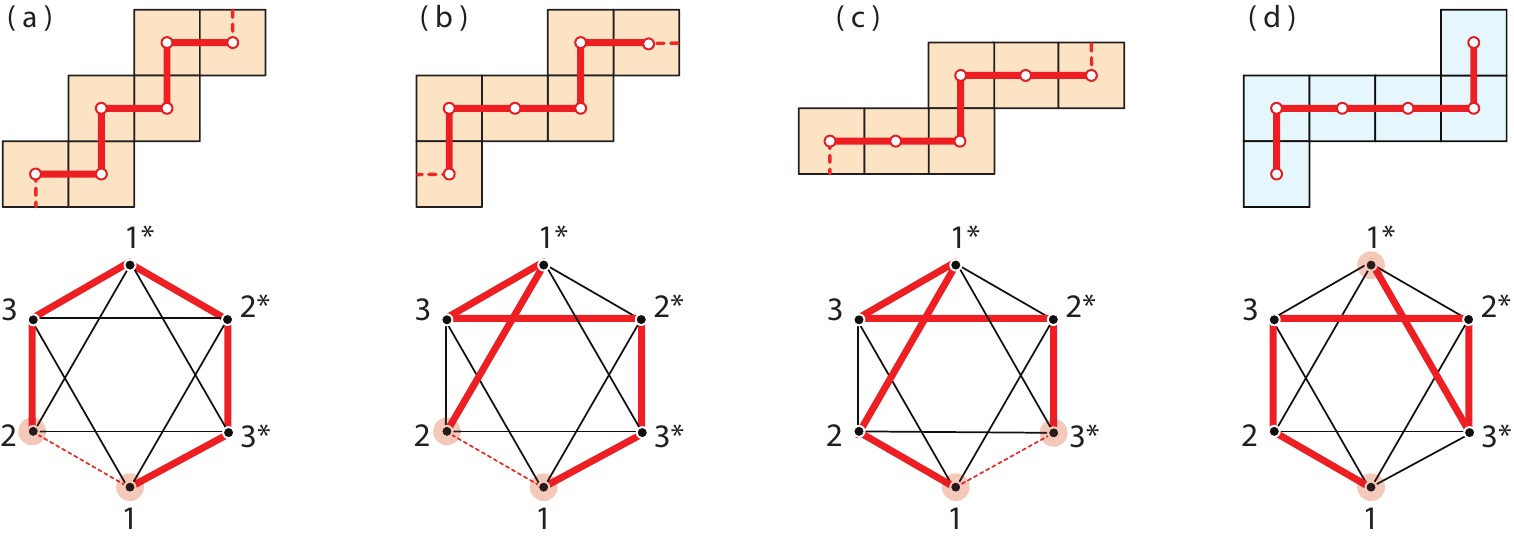}
\caption{Path nets for a 3-cube.}
\label{f:4paths}
\end{figure}

\begin{defn}
Path nets of the cube fall into two categories: \ter-paths that \emph{terminate} at antipodal positions of the Roberts graph, and \ext-paths that can \emph{extend} into spanning cycles by the addition of one edge.
\end{defn}


Our interest in studying \ext-paths is motivated by \emph{tract housing} from real estate development; see Cuff \cite{cuff} for historical context. Developers buy a large tract of land, subdivide it into individual small lots, and construct (near identical) homes on each lot based on a handful of predetermined variations, such as footprint, roof form, and material choices.  We (liberally) reinterpret spanning cycles on cubes as prefabricated structures that are shipped to tracts and unfolded on site into variations of homes.  Depending on the need, one can fabricate spanning cycles where the choice of deleted edge will yield different unfolded footprints, as in Figure~\ref{f:4paths}(bc), or the same footprint to compensate for contractor error, as in Figure~\ref{f:4paths}(a). As the dimension of the cube increases, there is a growth in the spectrum of available fabrication choices, as shown by  Table~\ref{t:combinatorics}.  For instance, when $n=4$, there are 7 distinct spanning cycles which result in 20 distinct  \ext-path nets.  Moreover, for $n>4$, Proposition~\ref{p:maxnet} claims the existence of a single spanning cycle on the $n$-cube which unfolds to $2n$ distinct (tract house) nets, depending on deleted edge.

Imagining the shipping of these prefabricated structures also brings up questions of rigidity.  Consider manufacturing an $n$-cube where all of its ridges are removed except for $2n$ of them, the edges of a spanning cycle in its Roberts graph.  Attach the $2n$ facets (along these $2n$ ridges) as \emph{hinges}, so their dihedral angles are not determined. We claim that such a cube remains  \emph{infinitesimally rigid}: there is no non-trivial first-order deformation that is an infinitesimal congruence on each facet \cite{gfa}. Thus, the prefabricated structure could be easily shipped to its tract housing location in a folded state.  With one additional cut of any ridge, however, the cube fully unfolds into an \ext-path net.  

\begin{conj}
All spanning cycles of the $n$-Roberts graph, corresponding to hinged spanning cycles of facets on the $n$-cube, produce an infinitesimally rigid structure.
\end{conj}


\subsection{}

Although the Roberts graph has been helpful in understanding the geometry of unfoldings through rolls, the language of chord diagrams, from the theory of Vassiliev knot and link invariants  \cite{bar}, is better suited to frame the study of spanning cycles.

\begin{defn}
A \emph{chord diagram} is a  matching of the nodes of a $2n$-gon. A chord is called a \emph{loop} if it pairs adjacent nodes of the polygon.  A \emph{k-loop} chord diagram has exactly $k$ loops, and is \emph{loopless} if it has none. Figure~\ref{f:chords} shows the set of seven unique loopless chord diagrams (up to symmetry)  for an 8-gon.
\end{defn}

\begin{figure} [h]
\includegraphics[width=\textwidth]{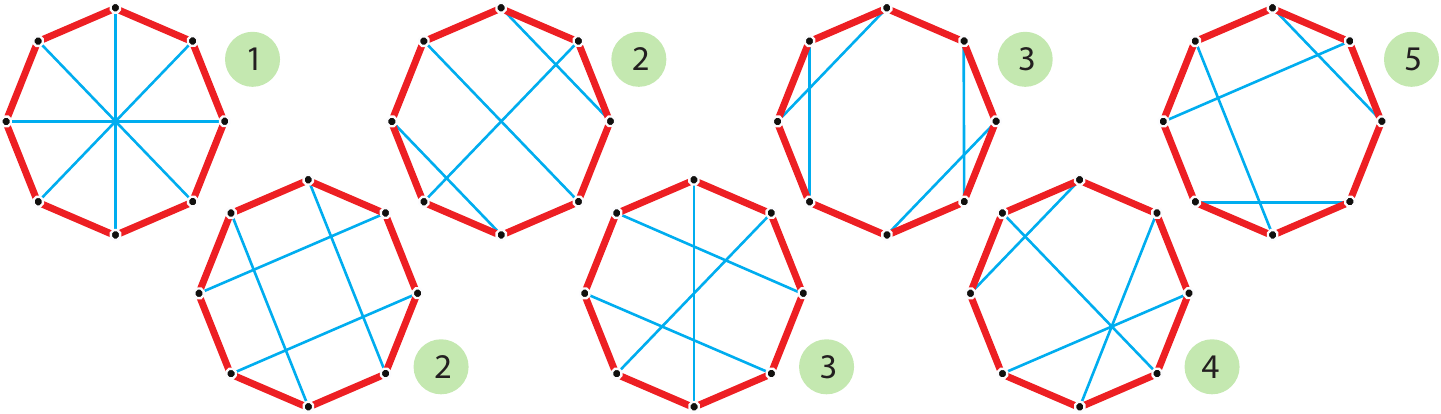}
\caption{Chord diagrams for eight nodes, along with the number of distinct \ext-path unfoldings associated to its spanning cycle.}
\label{f:chords}
\end{figure}

\begin{lem} \label{l:cycles}
There is a bijection between spanning cycles on the $n$-Roberts graph and loopless chord diagrams of a $2n$-gon.
\end{lem}

\begin{proof}
Consider a spanning cycle on the $n$-Roberts graph.  Excluding this $2n$-cycle, remove the remaining edges of the graph and add the $n$ edges connecting antipodal nodes.  Deform the spanning cycle into the boundary of a $2n$-gon, where the added edges naturally become chords of this diagram.  (Each column of Figure~\ref{f:chord-duality} shows an example, with the top row of spanning cycles deforming into the bottom row of loopless chord diagrams.) Since the spanning cycle cannot contain antipodal nodes of the Roberts graphs, only nonadjacent nodes will be pairwise matched.  This operation is invertible, and the result follows.
\end{proof}

\begin{figure}[h]
\includegraphics[width=\textwidth]{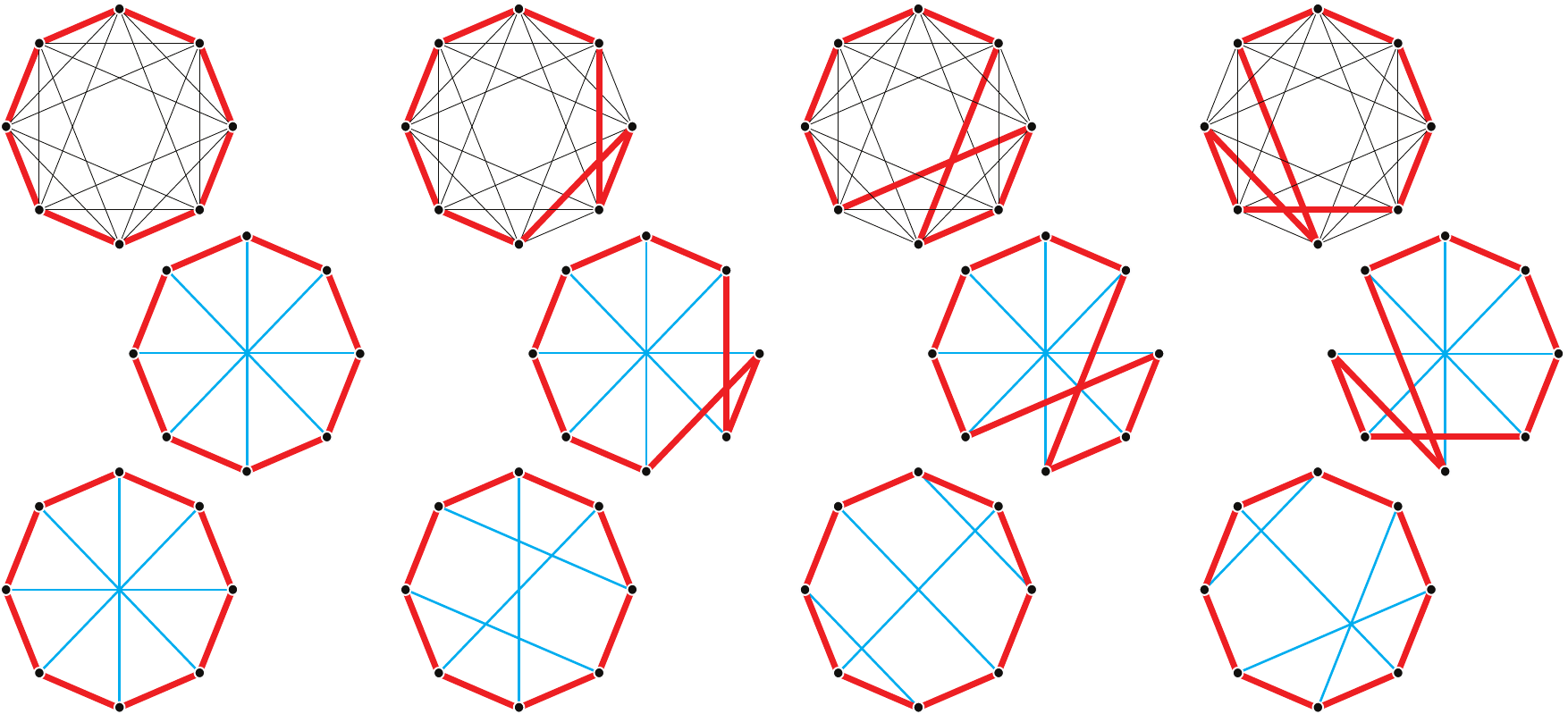}
\caption{Bijection between spanning cycles and loopless chord diagrams.}
\label{f:chord-duality}
\end{figure}

\begin{lem} \label{l:loops}
There is a bijection between \ter-paths on the $n$-Roberts graph and 1-loop chord diagrams of a $2n$-gon.
\end{lem}

\begin{proof}
A \ter-path terminates at antipodal positions of the cube and connecting these nodes with an edge forms a $2n$-cycle.  The chord diagram associated to this cycle has a unique loop between two adjacent nodes representing the antipodal positions.
\end{proof}

\begin{exmp}
Figure~\ref{f:4chords} shows the chord diagrams for path nets from Figure~\ref{f:4paths}.  The dashed  boundary lines correspond to the deleted edges needed to enable path unfoldings. 
\end{exmp}

\begin{figure}[h]
\includegraphics{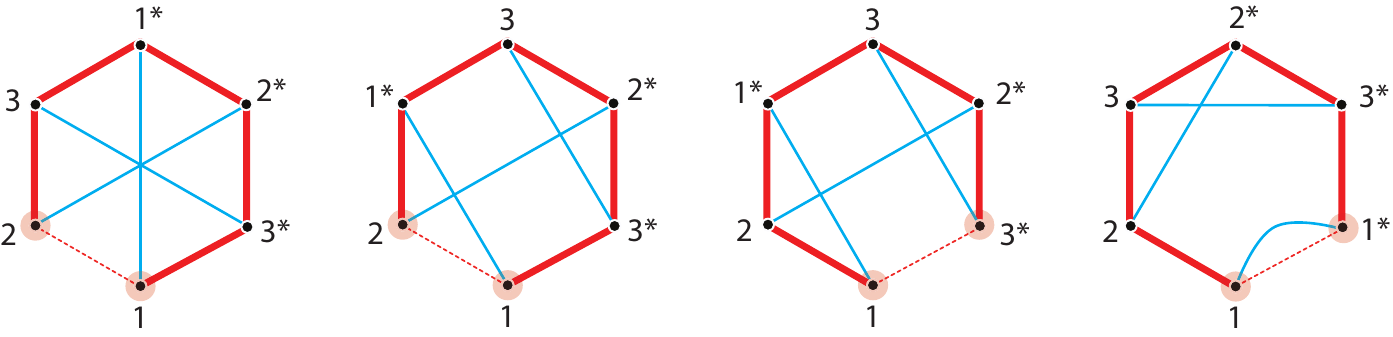}
\caption{Chord diagrams for spanning paths from Figure~\ref{f:4paths}.}
\label{f:4chords}
\end{figure}

\begin{lem}
The number of path unfoldings of an $n$-cube is equal to the number of \ter-paths on the $(n+1)$-Roberts graph.\end{lem}

\begin{proof}
Consider path unfoldings of an $n$-cube, viewed as spanning paths on the $n$-Roberts graph.  Add a \emph{marked edge} connecting the ends of these spanning paths to form cycles, reinterpreted as chord diagram of $2n$-gons; Figure~\ref{f:4chords} shows the $3$-cube version.  The loopless diagrams come from \ext-paths and the 1-loop diagrams from \ter-paths. Now add two vertices to the interior of the marked edge, connecting them with a chord;  Figure~\ref{f:chords-insert} shows this operation for the diagrams of Figure~\ref{f:4chords}.
\begin{figure} [h]
\includegraphics{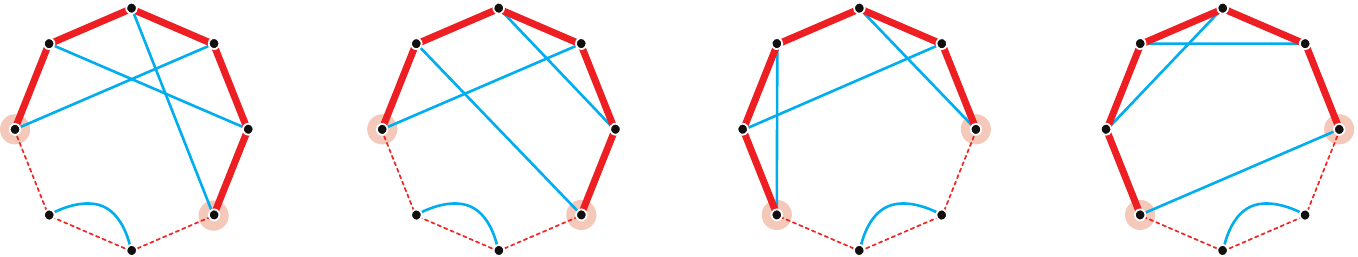}
\caption{The \ter-paths for the 4-cube.}
\label{f:chords-insert}
\end{figure}
The resulting collection is exactly the set of 1-loop chord diagrams of a $2(n+1)$-gon.  In particular, the marked edge of a $2n$ chord diagram coming from a \ter-path must be parallel to its loop edge, and inserting a new loop along this marked edge alters the previous loop to a regular chord.  Lemma~\ref{l:loops} finishes the bijection.
\end{proof}

\subsection{}

We close with some enumerative results.  The number of spanning cycles on the $n$-Roberts graph was originally studied by Singmaster \cite{sing} in 1975.  An explicit generating function is provided by Krasko and Omelchenko \cite{krom}  in 2017 by counting the loopless chord diagrams of a $2n$-gon, appearing as entry A003437 in the Online Encyclopedia of Integer Sequences (OEIS).  A generating function for the number of spanning paths on the $n$-Roberts graph (path unfoldings of the $n$-cube) is constructed by Burns \cite{bur} in 2016, appearing as entry A271215 in the OEIS.  Table~\ref{t:combinatorics} displays some of these values for spanning cycles and paths, the latter decomposing into \ter-paths and \ext-paths.

\renewcommand{\arraystretch}{1.5}{
\begin{table}[h]
\begin{center}
\begin{tabular}{lllll}
dimension \hspace{.2 in}  &  spanning cycles \hspace{.2 in} & path unfoldings \hspace{.2 in} & \ter-paths \hspace{.2 in} & \ext-paths \\
\hline
2 & 1 & 1 & 0 & 1 \\
3 & 2 & 4 & 1 & 3 \\
4 & 7 & 24 & 4 & 20 \\
5 & 29 & 184 & 24 & 160 \\
6 & 196 & 1911 & 184 & 1727 \\
7 & 1788 & 24252 & 1911 & 22341 \\
$n$ & A003437 & A271215 = $p_n$ & $p_{n-1}$ & $p_n - p_{n-1}$ \\
\end{tabular}
\end{center} \vspace{.1in}
\caption{Classification and enumeration of unfoldings of cubes.}
\label{t:combinatorics}
\end{table}
}

\begin{exmp}
For the $4$-cube, Table~\ref{t:combinatorics} shows 7 spanning cycles that unfold into 20 distinct  \ext-path nets.  Figure~\ref{f:chords} showcases these cycles as chord diagrams, where the displayed number counts the set of distinct \ext-path nets associated to its diagram.  Framed differently, it counts the unique number of tract homes that can be unfolded from a particular prefabricated cycle.  The sum of these numbers is 20, as expected.  
\end{exmp}

Given a spanning cycle, the number of distinct unfolded \ext-path nets depends on the symmetry of its chord diagram. Namely, a chord diagram that unfolds an $n$-cube into a unique net has maximal symmetry, with each deleted edge resulting in the same \ext-path unfolding; see Figure~\ref{f:chord-pattern}(a).  An unfolding into $2n$ distinct nets has no symmetry, resulting in a distinct \ext-path unfolding for each deleted edge; see Figure~\ref{f:chord-pattern}(d).  Enumeration based on arbitrary symmetries becomes difficult, but accessible for special cases:

\begin{prop} 
\label{p:maxnet}
Each $n$-cube (larger than dimension four) has four distinct spanning cycles, each of which unfolds into exactly 1, $\lceil n/2 \rceil$, $n$, or the maximum $2n$ distinct nets, respectively.
\end{prop}

\begin{proof}
Figure~\ref{f:chord-pattern} shows the chord arrangements for these four cases when  $n=6$. Part (a) is symmetric up to the full action of the dihedral group of order $2n$.  The three following parts are symmetric up to (b) reflection along both axes, (c) reflection only along the horizontal axis, and (d) neither rotation nor reflection.   To generalize the last three cases for larger values of $n$, insert additional pairwise vertices along the highlighted strips of the polygon, preserving the appropriate symmetries.  For $n=5$, delete a pair of antipodal nodes in (a) and horizontally collapse the highlighted strips into one edge for (bcd).
\end{proof}

\begin{figure}[h]
\includegraphics[width=.9\textwidth]{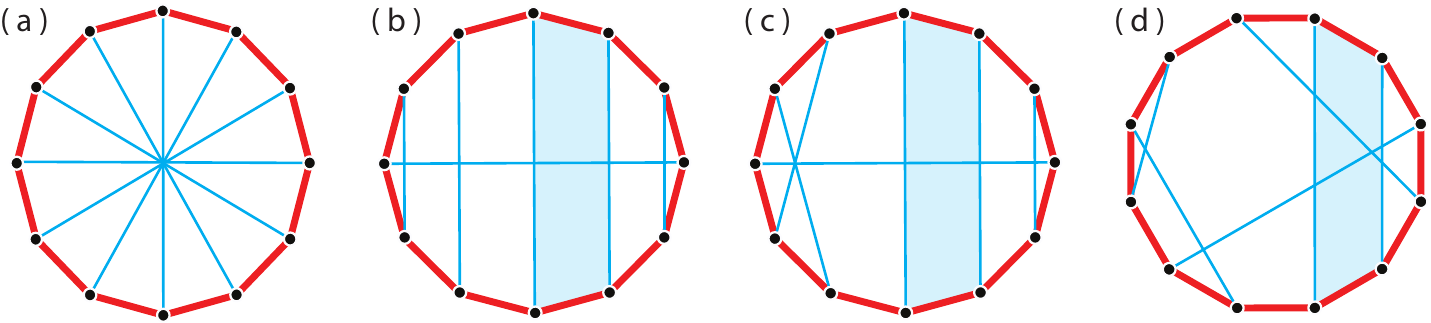}
\caption{Chord diagrams with 1, $\lceil n/2 \rceil$, $n$, and $2n$ unique nets, respectively.}
\label{f:chord-pattern}
\end{figure}

\begin{exmp}
From Proposition~\ref{p:maxnet}, the 5-cube will have four spanning cycles that unfold into exactly 1, 3, 5, and 10 distinct path nets.  Although there is only one spanning cycle that produces a unique path net unfolding, there are 8, 5, and 6 different spanning cycles that yield 3, 5, and 10 unique path nets, respectively. Figure~\ref{f:5-6-chords} shows these 6 spanning cycles, each of which unfold into 10 distinct nets. Moreover, due to chord diagram symmetries, there are no spanning cycles that unfold into 4, 7, 8, or 9 distinct nets, respectively. 
\end{exmp}

\begin{rem}
It would be interesting to explore the \emph{distribution} of these distinct unfoldings over all possible spanning cycles.
\end{rem}

\begin{figure}[h]
\includegraphics[width=\textwidth]{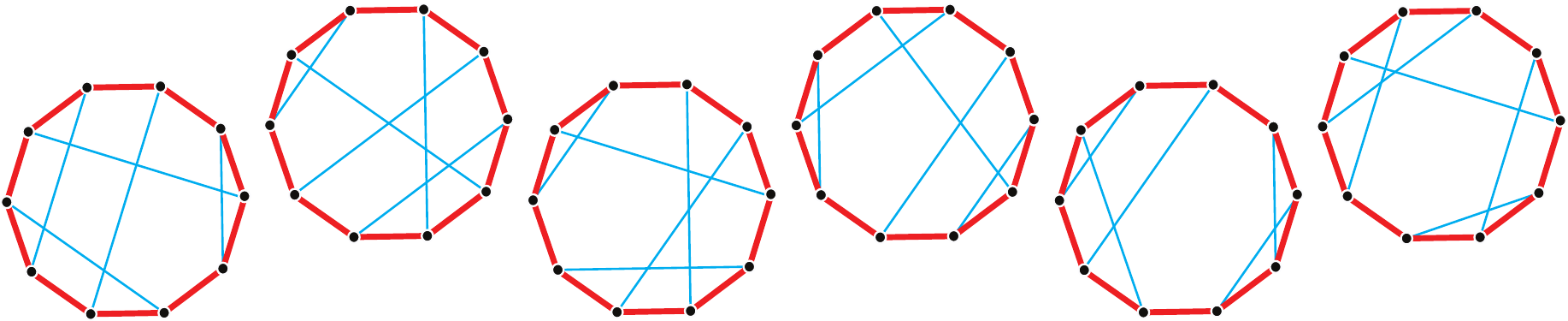}
\caption{Six spanning cycles, each of which unfold to 10 unique nets.}
\label{f:5-6-chords}
\end{figure}

%
%
\section{Conclusion}

The work of Horiyama and Shoji \cite{hosh} show that every edge unfolding of the five Platonic solids results in a net.  The higher-dimensional analogs of the Platonic solids are the regular convex polytopes: three classes of such polytopes exist for all dimensions (simplex, cube, orthoplex) and three additional  ones only appear in 4D (24-cell, 120-cell, 600-cell).  We have considered all unfoldings of cubes, and a similar result for simplices easily follows.  We are encouraged to claim the following:

\begin{conj}
Every ridge unfolding of a regular convex polytope yields a net.
\end{conj}

%
%
\bibliographystyle{amsplain}

\begin{thebibliography}{XXX}
\baselineskip=15pt

\bibitem{alex1} A.\ D.\ Aleksandrov. \emph{Intrinsic Geometry of Convex Surfaces}, Volume II (English translation) Chapman and Hall/CRC Press, 2005.

\bibitem{bar} D.\ Bar-Natan. On the Vassiliev knot invariants, \emph{Topology} {\bf 34} (1995) 423--472.

\bibitem{bupa} F.\ Buekenhout and M.\ Parker.  The number of nets of the regular convex polytopes in dimension $\leq 4$, \emph{Discrete Mathematics}  {\bf 186} (1998) 69--94.

\bibitem{bur} J.\ Burns. Counting a class of signed permutations and chord diagrams related to string matching with involutions, \emph{unpublished preprint}.

\bibitem{cuff} D.\ Cuff. Architects and builders in postwar Los Angeles, in \emph{Overdrive: L.A. Constructs the Future, 1940-1990}, Getty Research Institute (2013) 137--147.


\bibitem{gfa} E.\ Demaine and J.\ O'Rourke.  \emph{Geometric Folding Algorithms}, Cambridge University Press, 2007.

\bibitem{hosh} T.\ Horiyama and W.\ Shoji. Edge unfoldings of Platonic solids never overlap, \emph{Canadian Conference on Computational Geometry} (2011).

\bibitem{krom} E.\ Krasko and A.\ Omelchenko. Enumeration of chord diagrams without loops and parallel chords, \emph{Electronic Journal of Combinatorics} {\bf 24} (2017) P3.43.

\bibitem{mp} E.\ Miller and I.\ Pak. Metric combinatorics of convex polyhedra: cut loci and nonoverlapping unfoldings, \emph{Discrete and Computational Geometry} (2008) {\bf 39} 339--388.

\bibitem{sh2} G.\ Shephard. Angle deficiencies of convex polytopes, \emph{Journal of the London Mathematical Society} {\bf 43} (1969) 325--336.

\bibitem{sh1} G.\  Shephard. Convex polytopes with convex nets, \emph{Mathematical Proceedings of the Cambridge Philosophical Society} {\bf 78} (1975) 389--403.

\bibitem{sing} D.\ Singmaster, Hamiltonian circuits on the $n$-dimensional octahedron, \emph{Journal of Combinatorial Theory (B)} {\bf 19} (1975) 1--4.
 
\bibitem{tur} P.\ Turney. Unfolding the tesseract, \emph{Journal of Recreational Mathematics} {\bf 17} (1984) 1--16.

\end{thebibliography}

\end{document}